\documentclass[12pt]{amsart}
\usepackage{graphics,amsmath,amsfonts,amssymb,amscd,mathrsfs,pstricks,pst-plot}

\input{xy}
\xyoption{all}

\parskip=\smallskipamount

\newtheorem{theorem}{Theorem}[section]
\newtheorem{lemma}[theorem]{Lemma}

\newtheorem{proposition}[theorem]{Proposition}

\theoremstyle{definition}

\newtheorem{problem}[theorem]{Problem}

\definecolor{myblue}{rgb}{0.66,0.78,1.00}

\newcommand\hra{\hookrightarrow}

\newcommand{\cC}{\mathscr{C}}

\newcommand{\cO}{\mathscr{O}}

\newcommand{\C}{\mathbb{C}}

\newcommand{\N}{\mathbb{N}}

\newcommand{\Z}{\mathbb{Z}}
\renewcommand{\P}{\mathbb{P}}
\newcommand{\R}{\mathbb{R}}

\def\I{\mathrm{i}}

\newcommand\wt{\widetilde}
\newcommand{\Aut}{\mathop{{\rm Aut}}}

\numberwithin{equation}{section}

\begin{document}
\title%{Oka properties of $\C^n\setminus \R^k$}
{Fatou-Bieberbach domains in $\C^n\setminus\R^k$} 
%Oka properties of complements of totally real subspaces in a complex space}
\author{Franc Forstneri\v{c}}
\author{Erlend Forn\ae ss Wold}
\address{F.\ Forstneri\v c, Institute of Mathematics, Physics and Mechanics, 
University of Ljubljana, Jadranska 19, 1000 Ljubljana, Slovenia}
\email{franc.forstneric@fmf.uni-lj.si}
\address{E.\ F. Wold, Matematisk Institutt, Universitetet i Oslo,
Postboks 1053 Blindern, 0316 Oslo, Norway}
\email{erlendfw@math.uio.no}
%
%    General info
%
\date{\today}

\begin{abstract}
We construct Fatou-Bieberbach domains in $\C^n$ for $n>1$ which contain a given compact set $K$ and at the same time avoid a totally real affine subspace $L$ of dimension $<n$, provided that $K\cup L$ is polynomially convex. By using this result, we show that the domain $\C^n\setminus\R^k$ for $1\le k<n$ enjoys the basic Oka property with approximation for maps from any Stein manifold of dimension $<n$. 
\end{abstract}
\maketitle

%%%%%%%%%%%%%%%%%%%%%%%%%%%%%%%%%%%%%%%%%%%%%%%%%%%%%%%%%%%%%%%%%%%%%
%																		   %
%  INTRODUCTION   														   %
%																		   %
%%%%%%%%%%%%%%%%%%%%%%%%%%%%%%%%%%%%%%%%%%%%%%%%%%%%%%%%%%%%%%%%%%%%%
%
\section{Introduction}
\label{sec:Intro}
A proper subdomain $\Omega$ of a complex Euclidean space $\C^n$ is called a {\em Fatou-Bieberbach domain} if $\Omega$ is biholomorphic to $\C^n$.  
Such domains abound in $\C^n$ for any $n>1$; a survey can be found in \cite[Chap.\ 4]{F2011}. For example, an attracting basin of a holomorphic automorphism of $\C^n$ is either all of $\C^n$ or a Fatou-Bieberbach domain (cf.\ \cite[Appendix]{RR1988} or \cite[Theorem 4.3.2]{F2011}). Fatou-Bieberbach domains also arise as regions of convergence of sequences of compositions $\Phi_k=\phi_k\circ\phi_{k-1} \circ\cdots \circ\phi_1$, where each $\phi_j\in\Aut\C^n$ is sufficiently close to the identity map on a certain compact set $K_j\subset \C^n$ and the sets $K_j\subset \mathring K_{j+1}$ exhaust $\C^n$;  this is the so called {\em push-out method} (cf.\ \cite[Corollary 4.4.2, p.\ 115]{F2011}). 

Besides their intrinsic interest, Fatou-Bieberbach domains are very useful in constructions of holomorphic maps. 
% ; see for example \cite{G-S,F-R,F-W2009,F-W2013}. 
An important question is which pairs of disjoint closed subsets $K,L\subset \C^n$ of a Euclidean space 
of dimension $n>1$ can be separated by a Fatou-Bieberbach domain $\Omega$, in the sense that $\Omega$ contains one of the sets and is disjoint from the other one. Recently it was shown by Forstneri\v c and Ritter \cite{F-R} that this holds if the union $K\cup L$ is a compact polynomially convex set and one of the two sets is convex or, more generally, holomorphically contractible. They applied this to the construction of proper holomorphic maps of Stein manifolds of dimension $<n$ to $\C^n$ whose images avoid a given compact convex set. 

In this paper we prove the following similar result in the case when $L$ is an affine totally real subspace of 
$\C^n$ of dimension $<n$.

%
%
%   Fatou-Bieberbach domains separating a totally real subspace from a compact set
%
%
\begin{theorem}
\label{th:FB}
Let $n>1$. Assume that $L$ is a totally real affine subspace of $\C^n$ of dimension $\dim_\R L<n$ and $K$ is a compact subset of $\C^n\setminus L$. If $K\cup L$ is polynomially convex, then there exists a Fatou-Bieberbach domain $\Omega \subset \C^n$ with $K\subset \Omega\subset \C^n\setminus L$.
\end{theorem}

Theorem \ref{th:FB} is proved in \S \ref{sec:FB}. 

Concerning the hypotheses on $K\cup L$ in the theorem, recall that if $L$ is a closed unbounded subset of $\C^n$ (for example, a totally real affine subspace) and $K$ is a compact subset of $\C^n$, we say that $K\cup L$ is polynomially convex if $K\cup L'$ is such for every compact subset $L'$ of $L$. For results on such totally real sets see for example the paper \cite{L-W2009}.

Let $z=(z_1,\ldots,z_n)$, with $z_j=x_j+\I y_j$ and $\I=\sqrt{-1}$, denote the complex coordinates on $\C^n$. Any totally real affine subspace $L\subset\C^n$ of dimension $k\in \{1,\ldots,n\}$ can be mapped by an affine holomorphic automorphism of $\C^n$ onto the standard totally real subspace   
\begin{equation}\label{eq:Rk} 	
	\R^k = \{(z_1,\ldots, z_n) \in\C^n \colon y_1=\cdots = y_k=0,\ z_{k+1}= \cdots =z_n=0\}.
\end{equation}

The conclusion of Theorem \ref{th:FB} is false in general if $\dim L=n$ (the maximal possible dimension of a totally real submanifold of $\C^n$). An example is obtained by taking $L=\R^2\subset \C^2$ and $K$ the unit circle in the imaginary subspace $\I \R^2\subset \C^2$. The reason for the failure in this example is topological: the circle links
the real subspace $\R^2\subset \C^2$ 
(in the sense that it represents a nontrivial element of the fundamental group 
$\pi_1(\C^2\setminus \R^2)=\pi_1(\mathbb S^1)=\Z$), but  it would be contractibe to a point in any 
Fatou-Bieberbach domain containing it; hence such a domain can not exist. 
Such linking is impossible if $\dim_\R L<n$ since the complement $\C^n\setminus K$ of any compact polynomially convex set $K\subset \C^n$ is topologically a CW complex containing only cells of dimension $\ge n$, and hence it has vanishing homotopy groups up to dimension $n-1$ (cf.\ \cite{F1994,Hamm} or \cite[\S 3.11]{F2011}). 
Thus $L$ can be removed to infinity in the complement of $K$, so there is no linking.
(Another argument to this effect, using the flow of a certain vector field on $\C^n$, 
is given in the proof of Theorem \ref{th:FB} in Sec.\ \ref{sec:FB} below.)

In spite of this failure of Theorem \ref{th:FB} for $k=n$, $\C^n\setminus \R^n$ is known to be a union of Fatou-Bieberbach domains for any $n>1$; see Rosay and Rudin \cite{RR1988} or \cite[Example 4.3.10, p.\ 112]{F2011}. Therefore, the following seems a natural question.

\begin{problem}
Assume that $n>1$ and $K$ is a compact set in $\C^n\setminus\R^n$ such that $K\cup\R^n$ is polynomially convex and $K$ is contractible to a point in $\C^n\setminus \R^n$. Does there exist a Fatou-Bieberbach domain $\Omega$ in $\C^n$ satisfying $K\subset \Omega\subset \C^n\setminus \R^n$?
\end{problem}

Recall that a compact set $E$ in a complex manifold $X$ is said to be {\em $\cO(X)$-convex}
if for every point $p\in X\setminus E$ there exists a holomorphic function $f\in \cO(X)$ satisfying 
$|f(p)|>\sup_E |f|$. 

We shall apply Theorem \ref{th:FB} to prove the following second main result of the paper.

\begin{theorem} \label{th:main}
Assume that $L$ is an affine totally real subspace of $\C^n$ $(n>1)$ with $\dim_\R L<n$, 
$X$ is a Stein manifold with $\dim_\C X<n$, $E \subset X$ is a compact $\cO(X)$-convex set, 
$U\subset X$ is an open set containing $E$, and $f\colon U \to\C^n$ is a holomorphic map such that 
$f(E)\cap L=\emptyset$. Then $f$ can be approximated uniformly on $E$ by holomorphic maps 
$X\to \C^n\setminus L$.
\end{theorem}

In the language of {\em Oka theory} this means that maps $X\to \C^n\setminus L$ from Stein manifolds of dimension $<n$ to $\C^n\setminus L$ enjoy the basic Oka property with approximation (cf.\ \cite[\S 5.15]{F2011}). For Oka theory we refer to the monograph \cite{F2011}, the surveys \cite{FL1,FL3}, and the introductory note by L\'arusson \cite{L-Oka}.

If $X$ is a Stein manifold and $k\in \Z_+$ is an integer such that $\dim_\R X +k<2n$, 
then a generic holomorphic map $X\to \C^n$ avoids a given smooth submanifold $L\subset \C^n$ of real dimension 
$k$ in view of the Thom transversality theorem, so in this case Theorem \ref{th:main} trivially follows.
However, the transversality theorem does not suffice to prove Theorem  \ref{th:main} if 
$\dim_\R X + \dim_\R L \ge 2n$. The first nontrivial case is $\dim_\C X= \dim_\R L=2$, $n=3$.

Theorem \ref{th:main} is an immediate corollary to the following proposition.

\begin{proposition}
\label{prop:FB}
Assume that $L\subset \C^n$, $E\subset U\subset X$ and $f\colon U\to\C^n$ are as in Theorem \ref{th:main}, with $f(E)\cap L=\emptyset$. Then there exist an arbitrarily small holomorphic perturbation $\tilde f$ of $f$ on a neighborhood of $E$ and a Fatou-Bieberbach domain $\Omega=\Omega_{\tilde f} \subset \C^n$ such that 
\begin{equation} \label{eq:FB}
	\tilde f(E) \subset \Omega \subset \C^n\setminus L.
\end{equation}
\end{proposition}

Proposition \ref{prop:FB} is proved in \S  \ref{sec:maps}. The proof uses Theorem \ref{th:FB} together with the main result of the paper \cite{DF2010} by Drinovec Drnov\v sek and Forstneri\v c (see Theorem \ref{DFmain} below). 
Since $\Omega$ is (by definition) biholomorphic to $\C^n$, 
Theorem \ref{th:main} follows from Proposition \ref{prop:FB} 
by applying the Oka-Weil approximation theorem for holomorphic maps $X\to \Omega$.

In the setting of Theorem \ref{th:main}, the map $f$ extends (without changing its values on a neighborhood of $E$) to a continuous map $f_0 \colon X\to \C^n\setminus L$ by  topological reasons. In fact, the complement $\C^n\setminus L$ of an affine subspace $L\subset \C^n$ of real dimension $k$ is homotopy equivalent to the sphere $\mathbb S^{2n-k-1}$ and we have $2n-k-1\ge n$ by the hypotheses of the theorem, while $E$ has 
arbitrarily small smoothly bounded neighborhoods $V\subset X$ such that the pair $(X,\overline V)$ is 
homotopy equivalent to a relative CW complex of dimension at most $\dim_\C X <n$ 
(see \cite{Hamm} or \cite[p.~96]{F2011}).

Theorem \ref{th:main} is a first step in understanding the following problem.

\begin{problem}
Is $\C^n\setminus \R^k$ an Oka manifold for some (or for all) pairs of integers $1\le k\le n$ with $n>1$?
Equivalently, does the conclusion of Theorem \ref{th:main} hold for maps $X\to \C^n\setminus \R^k$ from all Stein manifolds $X$ irrespectively of their dimension?
\end{problem}

Since $\C^n\setminus \R^k$ is a union of Fatou-Bieberbach domains as mentioned above,
it is strongly dominable by $\C^n$, and hence is a natural candidate to be an Oka manifold. 
(A complex manifold $Z$ of dimension $n$ is {\em strongly dominable by $\C^n$} if for every point 
$z_0 \in Z$ there exists a holomorphic map $f\colon \C^n\to Z$ such that $f(0)=z_0$ and $f$ has maximal rank $n$ 
at $0$.)  Clearly every Oka manifold is strongly dominable, but the converse is not known. 
See e.g.\ \cite{FL1,FL2} for a discussion of this topic.
 
Our motivation for looking at this problem comes from two directions. One is that we currently know
rather few examples of open Oka subsets of complex Euclidean spaces $\C^n$ $(n>1)$.
A general class of such sets are complements $\C^n\setminus A$ of tame (in particular, of algebraic) 
complex subvarieties $A\subset \C^n$ of dimension $\dim A\le n-2$ (cf.\ \cite[Proposition 5.5.14, p.\ 205]{F2011}).
Furthermore, complements of some special low degree hypersurfaces are Oka.
In particular, the complement $\C^n\setminus \bigcup_{j=1}^k H_j$ of at most $n$ 
affine hyperplanes $H_1,\cdots, H_k\subset \C^n$ in general position is Oka
(Hanysz \cite[Theorem 3.1]{Hanysz}). Finally, basins of uniformly attracting sequences of
holomorphic automorphisms of $\C^n$ are Oka (Forn\ae ss and Wold \cite[Theorem 1.1]{FW}). 
On the other hand, 
the only compact sets $A\subset \C^n$ whose complements $\C^n\setminus A$ are known to be Oka are the finite sets. Recently Andrist and Wold \cite{AW2012} showed that the complement $\C^n\setminus B$ of a closed ball  
fails to be elliptic in the sense of Gromov  if $n\ge 3$ (cf.\ \cite[Def.\ 5.5.11, p.\ 203]{F2011}), but it remains unclear whether $\C^n\setminus B$ is Oka. (Ellipticity implies the Oka property, but the converse is not known.
A more complete analysis of holomorphic extendibility of sprays across compact subsets 
in Stein manifolds has been made recently by Andrist, Shcherbina and Wold \cite{ASW}.)
A positive step in this direction has been obtained recently in \cite{F-R} by showing that maps 
$X\to\C^n\setminus B$ from Stein manifolds $X$ with $\dim_\C X <n$ into the complement of a ball 
satisfy the Oka principle with approximation and interpolation.
 
Our second and more specific motivation was the question whether the set of Oka fibers is open in any holomorphic family of compact complex manifolds. (It was recently shown in \cite[Corollary 5]{FL2} that the set of Oka fibers fails to be closed in such families; an explicit recent example is due to Dloussky \cite{Dloussky}.) To answer this question in the negative, one could look at the example of Nakamura \cite{Nak} of a holomorphic family of compact three-folds such that the universal covering of the center fiber is $\C^3$, while the other fibers have coverings  biholomorphic to $(\C^2\setminus \R^2)\times \C$ (cf.\ page 98, Case 3 in \cite{Nak}). If $\C^2\setminus \R^2$ fails to be an Oka manifold then, since the class of Oka manifolds is closed under direct products and under 
unramified holomorphic coverings and quotients, the corresponding 3-fold (which is an unramified quotient of $(\C^2\setminus \R^2)\times \C$) also fails to be Oka; hence we would have an example of an isolated Oka fiber in a holomorphic family of compact complex threefolds.

%%%%%%%%%%%%%%%%%%%%%%%%%%%%%%%%%%%%%%%%%%%%%%%%%%%%%%%%%%%%%%%%%%%%%
%																		    %
%  A construction of Fatou-Bieberbach domains  						                           %
%																		    %
%%%%%%%%%%%%%%%%%%%%%%%%%%%%%%%%%%%%%%%%%%%%%%%%%%%%%%%%%%%%%%%%%%%%%
%
\section{A construction of Fatou-Bieberbach domains}
\label{sec:FB}
In this section we prove Theorem \ref{th:FB}. There is a remote analogy between this result and 
Proposition 11 in \cite{F-R}; the latter gives a similar separation of two compact disjoint sets in $\C^n$ whose union is polynomially convex and one of them is holomorphically contractible. In spite of this similarity, the proof of Theorem \ref{th:FB} is completely different and considerably more involved than the proof of the cited result from \cite{F-R}. 

We begin with some preliminaries. % We shall use the following notion.

If $L$ is a closed, unbounded, totally real submanifold of $\C^n$ and $K$ is a compact subset of $\C^n$, we say that $K\cup L$ is polynomially convex in $\C^n$ if $L$ can be exhausted by compact sets $L_1\subset L_2\subset\cdots\subset \bigcup_{j=1}^\infty L_j =L$ such that $K\cup L_j$ is polynomially convex for all $j\in\N$. If this holds then standard results imply that any function that is continuous on $L_j$ and holomorphic on a neighborhood of $K$ can be approximated, uniformly on $K\cup L_j$, by holomorphic polynomials on $\C^n$. It follows that $K\cup L'$ is then polynomially convex for any compact subset $L'$ of $L$. 
(See \cite{L-W2009} for these results.) 

We shall also need the following stability result. If $K$ and $L$ are as above, with $K\cap L=\emptyset$ and $K\cup L$ polynomially convex, then $K\cup \wt L$ is polynomially convex for any totally real submanifold $\wt L\subset \C^n$ that is sufficiently close to $L$ in the fine $\cC^1$ Whitney topology. (See L\o w and Wold \cite{L-W2009}.)

The proof of Theorem \ref{th:FB} hinges upon the
following recent result of Kutzschebauch and Wold \cite{KW2014} 
concerning Carleman approximation of certain isotopies of unbounded totally real submanifolds of $\C^n$ by holomorphic automorphisms of $\C^n$.

\begin{theorem}[Kutzschebauch and Wold, Theorem 1.1 in \cite {KW2014}] \label{th:KW2014}
Let $1\le k<n$ and consider $\R^k$ as the standard totally real subspace (\ref{eq:Rk}) of $\C^n$.
Let $K\subset\mathbb C^n$ be a compact set, let $\Omega\subset\C^n$ be an open set containing $K$, 
and let 
\[
	\phi_t \colon \Omega\cup\mathbb R^k \rightarrow\mathbb C^n,\quad t\in[0,1],
\]
be a smooth isotopy of embeddings, with $\phi_0=Id$, such that the following hold:
\begin{itemize}
\item[(1)] $\phi_t|_\Omega$ is holomorphic for all $t$,
\item[(2)] $\phi_t(K\cup\mathbb R^k)$ is polynomially convex for all $t$, and 
\item[(3)] there is a compact set $C\subset\mathbb R^k$ such that 
$\phi_t|_{\mathbb R^k\setminus C}=Id$ for all $t$.
\end{itemize}
Then for any $s\in\mathbb N$  the map $\phi_1$ is $\mathscr C^s$-approximable on  
$K\cup\mathbb R^k$, in the sense of Carleman, by holomorphic automorphisms of $\mathbb C^n$.
%
%The same is true if condition (3) is replaced by the following weaker condition:
%\begin{itemize}
%\item[(3')] there is some compact set $C\subset\mathbb R^k$ such that 
%$\phi_t|_{\mathbb R^k\setminus C}$ is sufficiently close to the identity in the fine $\mathscr C^s$
%topology on $\R^k\setminus C$ for all $t$.
%\end{itemize}
\end{theorem}

Theorem \ref{th:KW2014}  is a Carleman version of a theorem due to 
Forstneri\v c and L\o w \cite{FLow} (a special case was already proved in \cite{FRosay}).
The main difference is that, in Theorem \ref{th:KW2014}, the approximation takes place 
in the fine Whitney topology on {\em unbounded} totally real submanifolds of $\C^n$;
a considerably more difficult result to prove.
Results of this kind originate in the Anders\'en-Lempert-Forstne\-ri\v c-Rosay theorem 
\cite{AL,FRosay} on approximation of isotopies of biholomorphic maps
between Runge domains in $\C^n$ for $n>1$ by holomorphic automorphisms of $\C^n$.
(A survey can be found in \cite[Chap.\ 4]{F2011}.)

By using Theorem \ref{th:KW2014} we now prove the following lemma which is the
induction step in the proof of Theorem \ref{th:FB}.

\begin{lemma}\label{lem:inductive}
Assume that $L$ is a totally real affine subspace of $\C^n$ $(n>1)$ of dimension $\dim_\R L<n$ and
$K$ is a compact subset of $\C^n\setminus L$ such that $K\cup L$ is polynomially convex.
Given a compact polynomially convex set $M\subset \C^n$ with $K\subset \mathring M$ 
and numbers $\epsilon>0$ and $s\in \N$,
there exists a holomorphic automorphism $\Phi$ of $\C^n$ satisfying
the following conditions: 
\begin{itemize}
\item[\rm (a)] $\Phi(L)\cap M =\emptyset$, 
\item[\rm (b)] $\Phi(L)\cup M$ is polynomially convex, 
\item[\rm (c)] $\sup_{z\in K} |\Phi(z)-z| <\epsilon$, and 
\item[\rm (d)] $\Phi|_L$ is arbitrarily close to the identity in the fine $\mathscr C^s$ topology near infinity.
\end{itemize}
\end{lemma}

\begin{proof}
By a complex linear change of coordinates on $\C^n$ we may assume that $L$ is the standard totally real 
subspace $\R^k\subset \C^n$ given by (\ref{eq:Rk}). 

The properties of $K$ and $L$ imply that there exists a strongly plurisubharmonic Morse exhaustion function $\rho$ on $\C^n$ which is negative on $K$, positive on $L$, and equals $|z|^2$ near infinity. The first two properties are obtained in a standard way from polynomial convexity of $K\cup L$; for the last property see \cite[p.\ 299]{F1994}, proof of Theorem 1. 

Pick a closed ball $B\subset\C^n$ centered at the origin with $M\subset \mathring B$.
Let $V$ be the gradient vector field of $\rho$, multiplied by a smooth cutoff function of the form $h(|z|^2)$ that equals one near $B$ and equals zero near infinity. The flow $\psi_t$ of $V$ then exists for all $t\in \R$ and equals the identity near infinity. Apply this flow to $L$. 
Since $\dim_\R L < n$ and the Morse indices of $\rho$ are at most $n$
(as $\rho$ is strongly plurisubharmonic), a general position argument (deforming $V$ slightly if needed) 
shows that the trace of the isotopy $L_t:=\psi_t(L)$ for $t\ge 0$  does not approach any of the  (finitely many) 
critical points of $\rho$. It follows that $L$ is flown completely out of the ball $B$ in a finite time $t_0>0$, fixing $L$ all the time near infinity. Also, we have $K\cap L_t=\emptyset$ for all $t\in[0,1]$ by the construction. Reparametrizing the time scale, we may assume that this happens at time $t_0=1$; 
so $L_1$ is a smooth submanifold of $\C^n\setminus B$ which agrees with $L=\R^k$ near infinity. 

Since $L=L_0$ is totally real and the isotopy is fixed near infinity, Gromov's h-principle for totally real immersions 
(cf.\ \cite{Gromov:convex,Gromov:book}) implies that the isotopy $L_t$ constructed above
can be $\cC^0$ approximated by another isotopy $\wt L_t\subset \C^n$ consisting of totally real embeddings 
$\R^k\hra \C^n$, with $\wt L_0=L$ and $\wt L_t=L_t=\R^k$ near infinity for all $t\in [0,1]$. 
(Gromov's  theorem gives an isotopy of {\em totally real immersions}. Since $k<n$, 
a general position argument shows that a generic 1-parameter isotopy of immersions $\R^k\to \C^n$ consists 
of embeddings. Since we keep $\wt L_t$ fixed near infinity, we obtain totally real embeddings
provided that the approximation of the totally real immersions furnished by Gromov's theorem is close enough
in the $\mathscr C^1$ topology.)  
In particular, we may assume that the new isotopy $\wt L_t$ also avoids the set $K$ for all $t\in [0,1]$
and that $\wt L_1\subset \C^n \setminus B$. 

By Proposition 4.1 in \cite{KW2014} (see also L\o w and Wold \cite{L-W2009} for the nonparametric case)
there is a small smooth perturbation 
$\widehat L_t\subset \C^n$ of the isotopy $\wt L_t$ from the previous step, with $\widehat L_0=L$ and 
$\widehat L_t=\wt L_t=\R^k$ near infinity for all $t\in [0,1]$, such that $\widehat L_t$ is a totally real embedding 
and, what is the main new point, the union $K\cup \widehat L_t$ is polynomially convex for every $t\in [0,1]$. 
Furthermore, by the same argument we may assume that $\widehat L_t\cup B$ is polynomially convex 
for values of $t$ sufficiently close to $t=1$. % (since $\widehat L_t\cap B=\emptyset$ for such $t$). 

To simplify the notation we now replace the initial isotopy $L_t$ by $\widehat L_t$ and assume that 
the isotopy $L_t$ satisfies these properties. 

Pick a smooth family of diffeomorphisms $\phi_t\colon L\to L_t$ $(t\in [0,1])$ such that $\phi_0$ is the identity map on $L=L_0$ and $\phi_t$ is the identity on $L$ near infinity for every $t\in [0,1]$. 
Also we define $\phi_t$ to be the identity map on a neighborhood of $K$ for all $t\in [0,1]$;
this is possible since $K\cap L_t=\emptyset$ for all $t$ by the construction.

%We have now arrived at the main point of the proof: 

Under these  assumptions,  Theorem \ref{th:KW2014} furnishes a holomorphic automorphism 
$\Phi$ of $\C^n$ which is arbitrarily close to the identity map on $K$ 
and whose restriction to $L$ is arbitrarily close to the diffeomorphism $\phi_1\colon L\to L_1$ 
in the fine $\mathscr C^s$ topology on $L$. In particular, as $B\cup L_1$ is polynomially convex, 
we can ensure (by using stability of polynomial convexity mentioned at the beginning of this section) 
that $B\cup \Phi(L)$ is also polynomially convex. 
Since $M$ is polynomially convex and $M\subset B$, it follows that $M\cup \Phi(L)$ is 
polynomially convex as well.
\end{proof}

\begin{proof}[Proof of Theorem \ref{th:FB}]
We shall inductively apply Lemma \ref{lem:inductive} to construct a sequence of automorphisms
$\Psi_j$ of $\C^n$ which pushes $L$ to infinity and approximates the
identity map on $K$. The domain of convergence of this sequence will be a Fatou-Bieberbach 
domain  containing $K$ and avoiding $L$. (This is the so called  {\em push-out method};
see e.g.\ \cite[Corollary 4.4.2, p.\ 115]{F2011}.)

Pick an increasing sequence of closed balls 
$B=B_1\subset B_2\subset \cdots \subset \bigcup_{j=1}^\infty B_j =\C^n$ satisfying
$K\subset \mathring B_1$ and $B_j\subset \mathring B_{j+1}$ for every $j\in \N$. 
To begin the induction, let $\Phi_1$ be an automorphism of $\C^n$, 
furnished by Lemma \ref{lem:inductive} applied to $L$, $K$ and $M=B_1$, 
which satisfies the following  conditions:
\begin{itemize}
\item[\rm (a$_1$)]  $\Phi_1(L)\cap B_1=\emptyset$,
\item[\rm (b$_1$)]  $\Phi_1(L)\cup B_1$ is polynomially convex, and 
\item[\rm (c$_1$)]  $\Phi_1$ is close to the identity map on $K$ and $\Phi_1(K)\subset \mathring B_1$.
\end{itemize}
Property (b$_1$) implies that the set $\Phi_1^{-1}(\Phi_1(L)\cap B_1) = L\cup \Phi_1^{-1}(B_1)$ 
is polynomially convex, (a$_1$) shows that $L\cap \Phi_1^{-1}(B_1)=\emptyset$,
and (c$_1$) implies $K\subset  \Phi_1^{-1}(\mathring B_1)$. 

% Since the set $\Phi_1^{-1}(B_2)$ is polynomially convex and contains $\Phi_1^{-1}(B_1)$ in its interior, 
Next we apply Lemma \ref{lem:inductive} with $L$ as above, $M=\Phi_1^{-1}(B_2)$ and
$K$ replaced by $\Phi_1^{-1}(B_1)$ to find an automorphism $\wt \Phi_2$ of $\C^n$ satisfying 
the following  conditions:
\begin{itemize}
\item[\rm (a$_2$)] $\wt \Phi_2(L) \cap \Phi_1^{-1}(B_2) =\emptyset$, 
\item[\rm (b$_2$)] $\wt \Phi_2(L) \cup \Phi_1^{-1}(B_2)$ is polynomially convex, and 
\item[\rm (c$_2$)] $\wt \Phi_2$ is close to the identity map on $\Phi_1^{-1}(B_1)$.
\end{itemize}
Set $\Psi_1=\Phi_1$ and
\[
	\Phi_2=\Phi_1\circ \wt \Phi_2\circ \Phi_1^{-1},\qquad 
	\Psi_2=\Phi_2\circ \Phi_1 =\Phi_1\circ \wt \Phi_2.
\]
Property (a$_2$) implies that $\Psi_2(L)\cap B_2=\emptyset$, (b$_2$) implies that
$\Psi_2(L) \cup B_2$ is polynomially convex, and  (c$_2$) shows that 
$\Phi_2$ is close to the identity on $B_1$.

Continuing inductively we obtain a sequence of automorphisms 
$\Phi_j\in \Aut\C^n$ for $j=1,2,\ldots$ such that, setting 
$\Psi_j=\Phi_j\circ \Phi_{j-1}\circ\cdots\circ\Phi_1$, we have
\begin{itemize}
\item[\rm ($\alpha_j$)]     $\Psi_j(L)\cap B_j =\emptyset$ for each $j\in \N$, 
\item[\rm ($\beta_j$)]       $\Psi_j(L) \cup B_j$ is polynomially convex, and
\item[\rm ($\gamma_j$)]  $\Phi_j=\Psi_j\circ \Psi_{j-1}^{-1}$ is arbitrarily close to the identity map on 
$B_{j-1}$ for every $j>1$.
\end{itemize}
If $\Phi_j$ is chosen sufficiently close to the identity map on $B_j$ for every $j>1$ (and 
$\Phi_1$ is sufficiently close to the identity on $K$) then the domain of convergence
\[
	\Omega=\{z\in \C^n\colon \exists M_z>0 \ \text{such\ that}\ |\Psi_j(z)| \le M_z\ \ \forall j\in \N\}
\]
is a Fatou-Bieberbach domain \cite[Corollary 4.4.2]{F2011}.
Property ($\gamma_j$) ensures that $K\subset \Omega$, while property  ($\alpha_j$)
implies $L\cap \Omega=\emptyset$. This completes the proof of Theorem \ref{th:FB}.
\end{proof}

%%%%%%%%%%%%%%%%%%%%%%%%%%%%%%%%%%%%%%%%%%%%%%%%%%%%%%%%%%%%%%%%%%%%%
%																		    %
%  Holomorphic maps to $\C^n\setminus \R^k$   								            %
%																		    %
%%%%%%%%%%%%%%%%%%%%%%%%%%%%%%%%%%%%%%%%%%%%%%%%%%%%%%%%%%%%%%%%%%%%%
%
\section{Separation of varieties from totally real affine subspaces by Fatou-Bieberbach domains}  
\label{sec:maps}

In this section we prove Proposition \ref{prop:FB} stated in Sec.\ \ref{sec:Intro}. 
As explained there, this will also prove Theorem \ref{th:main}.

We shall  use the following result of Drinovec Drnov\v sek and Forstneri\v c which we state 
here  for future reference. We adjust the notation to the situation of the present paper.

%
%
%  MAIN THEOREM
%
%
\begin{theorem}[Theorem 1.1 in \cite{DF2010}]\label{DFmain}
Assume that $Z$ is an $n$-dimensional complex manifold, $\Omega$ is an open subset of $Z$, 
$\rho\colon \Omega\to (0,+\infty)$ is a smooth Morse function 
whose Levi form has at least $r$ positive eigenvalues at every point 
of $\Omega$ for some integer $r\le n$, and for any pair of real numbers $0<c_1<c_2$ the set 
\[
	\Omega_{c_1,c_2}= \{z\in\Omega\colon c_1\le \rho(z)\le c_2\}
\]
is compact. Let $D$ be a smoothly bounded, relatively compact, strongly pseudoconvex domain in a 
Stein manifold $X$, and let $f_0\colon\bar D\to Z$ be a continuous map 
that is holomorphic in $D$ and satisfies $f_0(bD)\subset \Omega$. If
\begin{itemize}
\item[(a)] $r\ge 2d$, where $d=\dim_{\C} X$, 
\end{itemize}
or if
\begin{itemize}
\item[(b)] $r\ge d+1$ and $\rho$ has no critical points of index $>2(n-d)$ in $\Omega$,
\end{itemize}
then $f_0$ can be approximated, uniformly on compacts in $D$, by holomorphic maps $f\colon D\to Z$ 
such that $f(x)\in \Omega$ for every $x\in D$ sufficiently close to $bD$ and 
\begin{equation}\label{eq:proper}
	\lim_{x\to bD} \rho(f(x))=+\infty.
\end{equation}
Moreover, given an integer $k\in\Z_+$, the map $f$ can be chosen to agree with 
$f_0$ to order $k$ at each point in a given finite set $\sigma\subset D$. 
\end{theorem}

The most important special case of this result, and the one that we shall use here, is when
$\rho\colon Z\to\R$ is an exhaustion function on $Z$; the condition (\ref{eq:proper}) then means that
the map $f\colon D\to Z$ is proper. In this case the proof of Theorem \ref{DFmain} in \cite{DF2010} 
also ensures that for every number $\delta>0$ we can pick $f$ as in the theorem such that
\begin{equation}\label{eq:smalldrop}
	\rho(f(x)) > \rho(f_0(x)) -\delta \quad \forall x\in D.
\end{equation}

\begin{proof}[Proof of Proposition \ref{prop:FB}] 
We may assume that $L=\R^k$ is the standard totally real subspace (\ref{eq:Rk}). Choose a number $R>0$ such that the set $f(E)\subset \C^n$ is contained in the ball $B_R=\{z\in\C^n\colon |z| <R\}$. 
Pick a smooth increasing convex function $h\colon\R\to\R_+$ such that $h(t)=0$ for $t\le 0$ and $h$ is strongly convex and strongly increasing on $t>0$. The nonnegative function on $\C^n$ given by
\begin{equation}\label{eq:rho}
	\rho(z_1,\ldots,z_n)=\sum_{j=1}^k y_j^2 + \sum_{j=k+1}^n |z_j|^2 + h\left(|z|^2-R^2\right) 
\end{equation}
is then a strongly plurisubharmonic exhaustion on $\C^n$ that vanishes precisely on the set $\overline B_R \cap L$.
It is easily seen that $\rho$ has no critical points on $\C^n\setminus (\overline B_R \cap L)=\{\rho>0\}$.

Since the compact set $E \subset X$ is $\cO(X)$-convex and $U\subset X$ is an open set containing $E$, 
we can pick a smoothly bounded strongly pseudoconvex domain $D$ with $E\subset D\Subset U$. 
Since $f(E)\subset B_R\setminus L$, we can ensure by choosing $D$ sufficiently small around $E$ 
that $f(\bar D)\subset B_R\setminus L$, and hence $c:=\min_{\bar D} \rho\circ f > 0$.  

By Theorem \ref{DFmain}, applied with the function $\rho$ (\ref{eq:rho}) on $Z=\C^n$,  
we can approximate the map $f$ uniformly on $E$
by a proper holomorphic map $\tilde f\colon  D \to \C^n$ satisfying 
\begin{equation}\label{eq:lowerbound} 
	\rho\circ \tilde f(x) \ge c/2>0 \quad \forall x\in  D. 
\end{equation}	
(Use case (b) of Theorem \ref{DFmain} and recall that $\rho$ is noncritical outside of its zero set.
Condition (\ref{eq:lowerbound})  follows from (\ref{eq:smalldrop}) applied with $\delta=c/2>0$.) 
Assuming that $\tilde f$ if close enough to $f$ on $E$, we have $\tilde f(E)\subset B_R\setminus L$. 
The image $A: =\tilde f(D)$ is then a closed complex subvariety of $\C^n$ which is disjoint from the set 
$\overline B_R \cap L=\rho^{-1}(0)$. 

Set $K=A\cap \overline B_R$. We claim that $K\cup L$ is polynomially convex. Since $\overline B_R \cup L$ is polynomially convex (see \cite[Theorem 8.1.26]{Stout2007}), we only need to consider points $p\in \overline B_R\setminus (A\cup L)$. Since $A$ is a closed complex subvariety of $\C^n$ and $p\notin A$, there is a function $g\in \cO(\C^n)$ such that $g(p)=1$ and $g|_A=0$. Also, since $p\notin L$, there exists a function $h\in \cO(\C^n)$ satisfying $h(p)=1$ and $|h|<1/2$ on $L$. (We use the Carleman approximation of the function that equals $1$ at $p$ and zero on $L$). The entire function $\xi = gh^N$ $(N\in\N)$ then satisfies $\xi(p)=1$ and $\xi=0$ on $A$, and it also satisfies $|\xi|<1$ on a given compact set $L'\subset L$ provided that the integer $N$ is chosen big enough. Hence $p$ does not belong to the polynomial hull of $K\cup L$, so this set is polynomially convex.

Since $\tilde f(E)\subset K$ by the construction, the existence of a Fatou-Bieberbach domain $\Omega$ satisfying 
$\tilde f(E) \subset \Omega \subset \C^n\setminus L$ (\ref{eq:FB}) now follows from Theorem \ref{th:FB} 
applied to the sets $K$ and $L$.  This completes the proof of Proposition \ref{prop:FB}.
\end{proof}

\smallskip
\textit{Acknowledgements.}
F.\ Forstneri\v c is partially supported by research program P1-0291 and research grant J1-5432 from ARRS, Republic of Slovenia. E.\ F.\ Wold is supported by grant  NFR-209751/F20 from the Norwegian Research Council.
A part of the work on this paper was done while Forstneri\v c was visiting the Mathematics Department of the 
University of Oslo. He wishes to thank this institution for its hospitality and excellent working conditions.

\bibliographystyle{amsplain}

\end{document}